\documentclass[12pt,centertags,oneside]{amsart}

\usepackage[mathscr]{eucal}

\RequirePackage[utf8]{inputenc}
\usepackage[english]{babel}
\usepackage{amsmath,amsfonts,amsthm,amssymb}\usepackage[top=3.5cm,left=2.5cm,right=2.5cm,bottom=3.5cm]{geometry}
\usepackage[usenames,dvipsnames]{color}
\usepackage[colorlinks=true,linkcolor=blue,citecolor=blue]{hyperref}

\usepackage{tikz}
\usepackage{multicol}
\usepackage{MnSymbol}
\usepackage{relsize}

\newcommand{\C}{\mathbb{C}}

\newcommand{\N}{\mathbb{N}}

\newcommand{\ddc}{\text{\normalfont dd}^c}
\newcommand{\diff}{\text{\normalfont d}}

\newcommand{\supp}{\text{\normalfont supp }}

\newcommand{\capK}{\rm cap}

\newcommand{\Curlywedge}{\mathlarger \curlywedge}

\newcommand{\codim}{\text{codim }}

\newtheorem{theorem}{Theorem}[section]
\newtheorem{proposition}[theorem]{Proposition}
\newtheorem{corollary}[theorem]{Corollary}
\newtheorem{lemma}[theorem]{Lemma}

\newtheorem*{theorem*}{Theorem}

\theoremstyle{definition}

\newtheorem*{definition*}{Definition}

\newtheorem{remark}[theorem]{Remark}

\numberwithin{equation}{section}

\usepackage{cite}

\title{Density and intersection of $(1,1)$-currents}

\author{Lucas Kaufmann, Duc-Viet Vu}
\address{Chalmers University of Technology - SE-412 96 Gothenburg, Sweden.}
\curraddr{Department of Mathematics, National University of Singapore, 10, Lower Kent Ridge Road, Singapore 119076}

\email{lucassa@chalmers.se; lucaskaufmann@nus.edu.sg}

\address{School of Mathematics, Korea institute for advanced study,
 85 Hoegiro, Dongdaemun-gu, Seoul 02455, Republic of Korea}
\email{vuviet@kias.re.kr}

\date{}

\begin{document}

\begin{abstract}
We study density currents associated with a collection of positive closed $(1,1)$-currents. We prove that the density current is unique and determined by the usual wedge product in some classical situations including the case where the currents have  bounded potentials. As an application, we compare density currents with the non-pluripolar product and the Andersson-Wulcan product. We also analyse some situations where the wedge product is not well-defined but the density can be explicitly computed.
\end{abstract}

\maketitle

\section{Introduction and main results} \label{sec1}

Let $X$ be a complex manifold of dimension $n$ and let $T_1,\ldots,T_m$ be positive closed currents on $X$. It is a central problem in complex analysis and its applications to give a meaning to the wedge product $T_1 \wedge \ldots \wedge T_m$. When the $T_j$ have bidegree $(1,1)$ the existence of local plurisubharmonic (psh) potentials makes the problem more tractable and in that case a good definition of $T_1 \wedge \ldots \wedge T_m$ can be given in many situations. See for instance \cite{Bedford_Taylor_76,demailly:lelong,Fornaess_Sibony}.  

When the $T_j$ have higher bidegree, a recent approach to this question was carried out by Dinh and Sibony in \cite{dinh-sibony:density}. They define the notion of density current associated with $T_1,\ldots,T_m$ that we briefly recall. 

 
 Consider the cartesian product $X^m$ and $$\mathbf T = T_1 \otimes \cdots \otimes T_m$$ as a positive closed current on $X^m.$
Denote by $\Delta \subset X^m$ the diagonal and by $N\Delta$ its normal bundle inside $X^m.$ Notice that $\Delta$ is naturally isomorphic to $X$.

An \textit{admissible map} $\tau: U \to W$ is a diffeomorphism from a neighbourhood $U$ of $\Delta$ inside $X^m$ to a neighbourhood $W$ of $\Delta$ seen as the zero section of $N\Delta$  such that $\tau$ restricts to the identity on $\Delta$ and its differential at $\Delta$ is the identity in the normal direction. These maps always exist but are not holomorphic in general. 

For  $\lambda \in \C^*$ let $$A_\lambda: N\Delta \to N\Delta$$ be given by fiberwise multiplication by  $\lambda$. A \textbf{density current} $R$ associated with $(T_1,\ldots,T_m)$ is a positive closed current on $N \Delta$ such that there exists a sequence of complex numbers  $\{\lambda_k\}_{k \in \N}$ converging to $\infty$ for which
$$R = \lim_{k \to \infty} (A_{\lambda_k})_* \tau_* \mathbf T,$$ for every admissible map $\tau.$ 
We say that the \textbf{Dinh-Sibony product} (or density product) of $T_1,\ldots,T_m$ exists if there is only one density current $R$ and $R = \pi^* S$ for some positive closed current on $\Delta$, where $\pi: N \Delta \to \Delta$ is the canonical projection. In that case we denote $$S = T_1 \Curlywedge \cdots \Curlywedge T_m.$$  
 
It was shown in \cite{dinh-sibony:density} that if $X$ is K\"ahler and the supports of $T_1, \ldots, T_m$ have a compact intersection, then density currents exist. 
Also, the cohomology class of the trivial extension of a density current $R$ to the projectivization $\overline{N \Delta}$ of $N \Delta$ is independent of  the sequence $\{\lambda_k\}_{k \in \N}$.  

Suppose now that  each $T_1,\ldots,T_{m-1}$ is of bidegree $(1,1)$, that $T:=T_m$ is of bidegree $(p,p)$ and  $m-1+p \leq n$. For $j=1,\ldots,m-1$ we can write locally $$T_j = \ddc u_j$$
for some psh functions $u_j$ which are unique up to the addition of a pluriharmonic function.

\begin{definition*}
We say that $(T_1,\ldots,T_{m-1},T)$ satisfy \textbf{Property $(\star)$} if the following holds

$(i)$ $u_{m-1}$ is locally integrable with respect to $T$ and for every $1 \le k \le m-2,$ the function $u_k$ is locally integrable with respect to $\ddc u_{k+1} \wedge \cdots \wedge \ddc u_{m-1} \wedge T$.

$(ii)$ for any open subset $U$ of $X$ and any sequence of smooth p.s.h.\ functions $u_{k}^j$ decreasing to $u_k$ on $U$ with $1 \le k \le m-2$ we have $$u^j_k\ddc u^j_{k+1} \wedge \cdots \wedge \ddc u^j_{m-1} \wedge T \longrightarrow u_k \ddc u_{k+1} \wedge \cdots \wedge \ddc u_{m-1} \wedge T \quad \text{on } U.$$
\end{definition*}

The last definition doesn't depend on the choice of the potentials $u_j$ of $T_j.$   Condition $(i)$ allows us to define $\ddc u_k \wedge \cdots \wedge \ddc u_{m-1} \wedge T$ recursively as
\begin{equation} \label{eq:BT-def}
\ddc(u_k \ddc u_{k+1} \wedge \cdots \wedge \ddc u_{m-1} \wedge T)
\end{equation}
and produces a  closed current. Hence we can define $T_k \wedge \cdots \wedge T_{m-1} \wedge T$ for every $k$, which is a closed current. Condition $(ii)$ implies that this current is positive. We shall say in this case that  $T_k \wedge \cdots \wedge T_{m-1} \wedge T$ is \textit{classically defined}.

Our first main result is the following.

\begin{theorem} \label{thm:compatible-def-wedge-general}
Let  $T_1, \ldots, T_{m-1},T$ be as above. Assume they satisfy Property $(\star).$
Then the Dinh-Sibony product of $T_1, \ldots, T_{m-1},T$ is well defined and
\begin{equation} \label{eq:compatible-def-wedge-general}
T_1 \wedge \cdots \wedge T_{m-1} \wedge T = T_1 \Curlywedge \cdots \Curlywedge T_{m-1} \Curlywedge T.
\end{equation}
\end{theorem}

The question of determining whether a given collection $(T_1,\ldots,T_{m-1},T)$ satisfies Property $(\star)$ has been intensively studied. In particular, this is the case when $u_1, \ldots, u_{m-1}$ are locally bounded.  When $T$ is the constant function equal to one then $(i)$ and $(ii)$ are known to hold for a large class of psh functions. An important particular case is when the intersection of the unbounded loci of the $u_j$ has small Hausdorff dimension (see \cite[Prop. 3.6]{demailly:agbook} or \cite[Cor. 3.6]{Fornaess_Sibony}). We refer to \cite{Chern_Levine_Nirenberg,Bedford_Taylor_76,Fornaess_Sibony,demailly:agbook} for more detailed information. 

It is worth noting that in the setting of Theorem \ref{thm:compatible-def-wedge-general},  we don't require neither the compactness of the intersection of the supports nor the condition that $X$ is K\"ahler. In that case, the existence of a density current (and its uniqueness) will follow from our proof. 

In some situations it is not possible to define an intersection product using (\ref{eq:BT-def}). Nevertheless we can still define what is called the non-pluripolar product as introduced by Bedford-Taylor and Boucksom-Eyssidieux-Guedj-Zeriahi in  \cite{BT_fine_87,BEGZ} (see Section \ref{sec_nonpluri} for the definition). When the currents have analytic singularities this product is related to the density in the following way.

\begin{theorem} \label{the_pluridden}
Let $T_1, \ldots, T_m$ be positive closed $(1,1)-$currents with analytic singularities whose supports have compact intersection. Denote by $Z_j$ the singular locus of $T_j$ and let $Z = \bigcup_{j=1}^m Z_j$. Then  the non-pluripolar product $ \langle T_1 \wedge \cdots \wedge T_m \rangle$ is well-defined and every density current $S$ associated with $(T_1, \ldots, T_m)$ satisfies
\begin{equation*}
\mathbf{1}_{\pi^{-1}(X \setminus Z)} S = \pi^* \langle T_1 \wedge \cdots \wedge T_m \rangle.
\end{equation*}
In particular $\pi^*  \langle T_1 \wedge \cdots \wedge T_m \rangle \leq S$ and if the Dinh-Sibony product of $T_1 , \ldots, T_{m} $ is well defined we have $ \langle T_1 \wedge \cdots \wedge T_m \rangle \leq T_1 \Curlywedge \cdots \Curlywedge T_m $.
\end{theorem}

The above result can be proved using Theorem \ref{thm:compatible-def-wedge-general} but we also have a more general statement for arbitrary currents, given by Theorem \ref{th_density_nonpluri}  below.  

Given a positive closed $(1,1)$-current $T$ on $X$ with analytic singularities, Andersson-Wulcan gave a meaning to the $m$-fold self-product of $T$ for every $m=1,\ldots,n$, denoted by $T^{m}_{AW}$ (see   \cite{andersson-wulcan} and Section \ref{sec:AW}). We have the following comparison result.

\begin{theorem} \label{th_AWdensity}  Let $T$ be a positive closed $(1,1)$-current on $X$ with analytic singularities. If $S$ is a density current associated with the $m-$tuple $(T,\ldots,T)$ then
\begin{equation} \label{eq:AW-ineq}
\pi^*T^{m}_{AW} \le S.
\end{equation}
\end{theorem}

The paper is organized as follows. In Section \ref{sec2}, we prove Theorem \ref{thm:compatible-def-wedge-general}. Theorems \ref{the_pluridden} and \ref{th_AWdensity} are proved in Sections \ref{sec_nonpluri} and \ref{sec_nonproper} respectively. 

\textbf{Acnowledgements:} The first named author would like to thank Elizabeth Wulcan for fruitful discussions that served as inspiration for some of the results in this paper.


\section{Density product vs.\ classical product} \label{sec2}

This section is devoted to the proof of Theorem \ref{thm:compatible-def-wedge-general}.

Since our problem is of a local nature, we can assume that $X$ is a ball in $\C^n$ and $T_j := \ddc u_j$ on $X$ for $j=1,\ldots,m-1$.  A neighbourhood of $\Delta$ inside $X^m$ identifies with that of the zero section of the trivial bundle $\pi: (\C^n)^{m-1} \times X \to X$ via the change of coordinates
\begin{equation}
\label{eq:standard-tau}
\varrho(x^1, \ldots, x^m) = (x^1-x^m, \ldots, x^{m-1} -x^m, x^m) :=(y^1, \ldots, y^m) := (y',y^m):= y.
\end{equation}
By definition, any  admissible map $\tau$ is given by 
$$\tau(y',y^m)= \big(y', y^m + y' a(y)\big)+ O(|y'|^2),$$
where $a(y)$ is a matrix whose entries are complex smooth functions. Our results as well as the arguments below in fact holds for every bi-Lipschitz admissible map $\tau.$ But for simplicity, we only consider smooth admissible maps.              

Notice that \begin{equation}
\varrho_* (T_1 \otimes \cdots \otimes T_m) = \ddc \widetilde u_1 \wedge \cdots \wedge \ddc \widetilde u_{m-1} \wedge \widetilde T , 
\end{equation}
where $\widetilde u_j(y',y^m) =u_j(y^j + y^m)$ for $j=1,\ldots,m-1$ and $\widetilde T = \pi^* T$, $\pi(y',y^m)=y^m$.

Our first step to  prove Theorem \ref{thm:compatible-def-wedge-general} is to show that with this particular choice of admissible coordinate the desired result holds.

\begin{proposition} \label{pro_voitauthoi}
\begin{equation*} 
(A_\lambda)_* \big(\ddc \widetilde u_1\wedge \cdots \wedge \ddc \widetilde u_{m-1} \wedge \widetilde T\big) \, \,  \stackrel{|\lambda| \to \infty}{\longrightarrow}  \pi^*(\ddc u_1 \wedge \cdots \wedge \ddc u_{m-1} \wedge T). 
\end{equation*}
\end{proposition}
We will need the following lemma. 

\begin{lemma} \label{le_hoithulambda2} Let $R$ be a positive closed current on $(\C^n)^{m-1} \times X$ and let $\lambda_k \to \infty$ be such that $(A_{\lambda_k})_*R$ converges weakly to $\pi^*S_\infty$ for some current $S_\infty$ on $X$. Let $u_j$ and $\widetilde u_j$ be as above and suppose that $u_j$ is locally integrable with respect to $S_\infty$. If  $R_j$ is a limit point of the sequence $(A_{\lambda_k})_*(\widetilde{u}_j R)$ as $\lambda_k \rightarrow \infty$, then $R_j \le \pi^*( u_j  S_\infty)$.   
\end{lemma}

\begin{proof} This is a consequence of \cite[Prop. 3.2]{Fornaess_Sibony}. We give here a direct proof for readers' convenience.   

Let $u_{j,\epsilon}$ be a regularization of $u_j$ by convolution such that $u_{j,\epsilon}$ decreases pointwise to $u_j$ as $\epsilon \to 0$.  Then, for fixed $\lambda$, we have $(A_\lambda)_*\widetilde{u}_{j,\epsilon}  \searrow (A_\lambda)_*\widetilde{u}_{j}$ as $\epsilon \rightarrow 0$,  where $\widetilde{u}_{j,\epsilon}(y^1, \ldots, y^m):= u_{j,\epsilon}(y^j+y^m).$ 
Since the $u_{j,\epsilon}$ are continuous we have, for fixed $\epsilon$, that $(A_\lambda)_*\widetilde{u}_{j,\epsilon} \rightarrow  \pi^* u_{j,\epsilon}$ locally uniformly as $\lambda \rightarrow \infty$. 

By the above convergence and the fact that $(A_{\lambda_k})_*R \to \pi^*S_\infty$, we have $$\lim_{k \rightarrow \infty} (A_{\lambda_k})_*(\widetilde{u}_{j,\epsilon} R) =  \pi^*(u_{j,\epsilon} S_\infty)$$ for each $\epsilon$. Combined with the fact that $(A_{\lambda_k})_*  (\widetilde{u}_{j,\epsilon} R) \ge (A_{\lambda_k})_* (\widetilde{u}_{j} R)$ we get that $ R_j\le \pi^* (u_{j,\epsilon} S_\infty)$
for every $\epsilon.$ Taking $\epsilon \rightarrow 0$ yields the desired result.
\end{proof}

Our proof of Proposition \ref{pro_voitauthoi} will be by induction on $m$, so let
\begin{equation} \label{eq:Rm}
R_m:= \ddc_{y^m} \widetilde{u}_2 \wedge \cdots \wedge \ddc_{y^m} \widetilde{u}_{m-1} \wedge \pi^* T,
\end{equation}

where $\ddc_{y^m}$ means that we only consider (weak) derivatives with respect to the $y^m$ variable. In order to see that $R_m$ is well defined we first notice  that $\widetilde{u}_{m-1} \pi^* T$ has locally finite mass. Indeed, for a compact $K = K_1 \times K_2 \subset (\C^n)^{m-1} \times X$ and a test form $\Phi = \Phi_1(y') \wedge \Phi_2(y^m)$ supported in $K$ we have
\begin{align*}
\int_K \widetilde{u}_{m-1} \pi^* T \wedge \Phi = \int_{K_2} \left( \int_{K_1} u_{m-1}(y^{m-1} + y^m)  \wedge \Phi_1(y')\right)  \pi^* T \wedge \Phi_2(y^m)
\end{align*}
and the inner integral on the right hand side is a smooth function of $y^m$ (since it is a convolution). Hence the left hand side is finite.

Therefore $\widetilde{u}_{m-1} \pi^* T$ is well defined as a current in $(\C^n)^{m-1} \times X$ and  we can define  $\ddc_{y^m}\widetilde{u}_{m-1} \wedge \pi^* T := \ddc_{y^m}(\widetilde{u}_{m-1} \pi^* T)$. Applying the same procedure recursively allows us to show that $R_m$ is well defined.

Notice that $R_m$ is a positive current (in particular it is of order zero), but not necessarily closed.  For $1 \le  j \le m-1,$ let $u_{j,\epsilon}, \widetilde{u}_{j,\epsilon}$ be as in the proof of Lemma \ref{le_hoithulambda2}. Put 
$$R_{m,\epsilon}:= \ddc_{y^m} \widetilde{u}_{2,\epsilon} \wedge \cdots \wedge \ddc_{y^m} \widetilde{u}_{m-1,\epsilon} \wedge \pi^* T.$$
The definition of $R_m$ imply that $\widetilde{u}_1$ is locally integrable with respect to $R_m$ and 
\begin{align} \label{lim_Rmeo}
\widetilde{u}_{1, \epsilon} R_{m,\epsilon}  \longrightarrow \widetilde{u}_1 R_m
\end{align}
as $\epsilon \to 0.$ 

\begin{lemma}\label{le_claim1}  Let $R_m$ be as in \emph{(\ref{eq:Rm})}. Then
\begin{align} \label{lim_Tm2}
 (A_\lambda)_*(\widetilde{u}_1 R_m) \rightarrow  \pi^* \left( u_{1}  \ddc u_2 \wedge  \cdots \wedge  \ddc u_{m-1} \wedge T \right)
\end{align}
as currents on $(\C^n)^{m-1} \times X$.
\end{lemma}

\begin{proof}
Notice that the left hand side in (\ref{lim_Tm2}) has degree zero in $y'$, so in order to prove the above convergence it suffices to test $(A_\lambda)_*(\widetilde{u}_1 R_m)$ against forms of the type  $\Phi = \Phi_1(y') \wedge \Phi_2(y^m)$ where $\Phi_1$ has full bidegree $(n(m-1),n(m-1))$.

Let $\chi$ be a positive smooth radial function with compact support in $\C^n$ such that $\int_{\C^n}  \chi \cdot i^n \,  d y \wedge d \bar{y}=1.$ Put 
$$\widetilde{\chi}(y^1, \ldots, y^{m-1}):= \chi(y^1) \cdots \chi(y^{m-1}).$$
Let 
$\Phi_{1,\chi}(y'): =  i^{n(m-1)} \widetilde{\chi}(y') \diff y' \wedge \diff \overline{y'}$ and  $\Phi_2$ be a form in $y^m$  with compact support in $X$. Put $\Phi_{\chi}:= \Phi_{1,\chi} \wedge \Phi_2.$ 
We have $\pi_* (\Phi_\chi) = \Phi_2$ and $\langle S, \Phi_2 \rangle = \langle \pi^* S,\Phi_\chi \rangle$ for any $(m-2+p,m-2+p)$-current $S$ on $X$. Put  
$$ u_{j,\epsilon}^\lambda (y^m) =  \int_{y^j} u_{j,\epsilon}(\lambda^{-1}y^j + y^m) \chi(y^j) (i^n \diff y^j \wedge \diff \overline{y^j})$$
 for $j=1,\ldots,m-1$ and
$$R_{\lambda,\epsilon} =  u_{1,\epsilon}^\lambda  \, \ddc u_{2,\epsilon}^\lambda \wedge  \cdots \wedge  \ddc u_{m-1,\epsilon}^\lambda \wedge T$$
as a current  in $(\C^n,y^m).$  Let $u_j^\lambda, R_\lambda$ be defined in the same way as $u_{j,\epsilon}^\lambda, R_{\lambda,\epsilon}$ respectively but with  $u_{j,\epsilon}$ replaced by $u_j$. Observe that the defining formula of $u_j^\lambda$ is a convolution with a positive smooth radial function, hence $u_j^\lambda$ is smooth and decreases to $u_j$ as $|\lambda| \to \infty$. Moreover, $u_{j,\epsilon}^\lambda$ is also smooth and   for $\lambda$ fixed,   $u_{j,\epsilon}^\lambda$ is smooth and  decreases to $u_j^\lambda$ as $\epsilon \to 0.$      We have
\begin{align*}
 (A_\lambda)_*(\widetilde{u}_{1,\epsilon} R_{m,\epsilon}) = \widetilde{u}_{1,\epsilon} \circ A_\lambda^{-1} \ddc_{y^m} (\widetilde{u}_{2,\epsilon} \circ A_\lambda^{-1}) \wedge \cdots \wedge \ddc_{y^m} (\widetilde{u}_{m-1,\epsilon} \circ A_\lambda^{-1}) \wedge \pi^* T.
\end{align*}
Applying Fubini's Theorem to the last equality gives
\begin{align*}
\langle (A_\lambda)_*(\widetilde{u}_{1,\epsilon} R_{m,\epsilon}), \Phi_\chi   \rangle =\langle R_{\lambda,\epsilon}, \Phi_2 \rangle.  
\end{align*}
Letting $\epsilon \to 0$ in the two sides of the last equality and using (\ref{lim_Rmeo}), we obtain that 
\begin{align} \label{eq_AlmabRkalmm}
\langle (A_\lambda)_*(\widetilde{u}_{1} R_{m}), \Phi_\chi   \rangle =\langle R_{\lambda}, \Phi_2 \rangle. 
\end{align}

 By hypothesis, $(T_1,\ldots, T_{m-1},T)$ satisfy Property ($\star$), so $$R_\lambda \to  u_{1}  \ddc u_2 \wedge  \cdots \wedge  \ddc u_{m-1} \wedge T$$  as $|\lambda| \to \infty$ because the smooth p.s.h. functions $u_j^\lambda$ decreases to $u_j$.  Combining this with (\ref{eq_AlmabRkalmm}) yields
\begin{align} \label{lim_lamRmpho}
 \langle (A_\lambda)_*(\widetilde{u}_1 R_m), \Phi_\chi  \rangle = \langle R_\lambda, \Phi_2 \rangle \longrightarrow & \langle u_{1}  \ddc u_2 \wedge  \cdots \wedge  \ddc u_{m-1} \wedge T, \Phi_2 \rangle \\ 
\nonumber
=  & \langle \pi^*(u_{1}  \ddc u_2 \wedge  \cdots \wedge  \ddc u_{m-1} \wedge T), \Phi_\chi \rangle.
\end{align}

In particular, this implies that the family $\{(A_\lambda)_*(\widetilde{u}_1 R_m)\}_{|\lambda| \geq 1}$ is of bounded mass norm on  fixed compact subsets of $(\C^{n})^{m-1} \times X$ uniformly in $\lambda.$ Combined with (\ref{lim_lamRmpho}) and the fact  that the forms $\Phi_\chi$ span a dense subspace in the set of test forms, this gives (\ref{lim_Tm2}). This  finishes the proof.
\end{proof}

\begin{proof}[Proof of Proposition \ref{pro_voitauthoi}] Recall that we need to show that
\begin{equation} \label{eq:density-convergence}
(A_\lambda)_* \big( \ddc \widetilde u_1\wedge \cdots \wedge \ddc \widetilde u_{m-1} \wedge \widetilde T \big) \,\, \stackrel{|\lambda| \to \infty}{\longrightarrow} \pi^*(\ddc u_1 \wedge \cdots \wedge \ddc u_{m-1} \wedge T). 
\end{equation}

We will do so by induction. For $m=1$ there is nothing to be shown. Suppose now that  (\ref{eq:density-convergence}) holds for $(u_2,\ldots,u_{m-1},T)$. This means that
\begin{equation} \label{eq:induction-step}
(A_\lambda)_*\big(\ddc \widetilde{u}_2 \wedge \cdots \wedge \ddc \widetilde{u}_{m-1} \wedge \widetilde T\big)  \stackrel{|\lambda| \to \infty}{\longrightarrow}  \pi^*(\ddc u_2 \wedge  \cdots \wedge  \ddc u_{m-1} \wedge T).
\end{equation}

Let $R_m$ be the current defined in (\ref{eq:Rm}). We claim that
\begin{align} \label{eq_limitAlmabdaddcu1loaitru}
\lim_{\lambda \rightarrow \infty} (A_\lambda)_* \big[\widetilde{u}_1 \ddc \widetilde{u}_2 \wedge \cdots \wedge \ddc \widetilde{u}_{m-1} \wedge \widetilde T -  \widetilde{u}_1 R_m\big]=0.
\end{align}
Indeed,  let $\nu_\infty$ be a limit point of $(A_\lambda)_*\big(\widetilde{u}_1 \ddc \widetilde{u}_2 \wedge \cdots \wedge \ddc \widetilde{u}_{m-1} \wedge \widetilde T \big)$. Recall that $u_1$ is assumed to be locally integrable with respect to  $\ddc u_2 \wedge  \cdots \wedge  \ddc u_{m-1} \wedge T$. By Lemma \ref{le_hoithulambda2} and (\ref{eq:induction-step}),  we have  
\begin{align} \label{ine_nuinfty}
\nu_\infty \le \pi^*(u_1 \ddc u_2 \wedge  \cdots \wedge  \ddc u_{m-1} \wedge T).
\end{align}
 Let  $\Phi(y', y^m) =\Phi_1(y') \wedge \Phi_2(y^m)$ be a weakly positive test form on $(\C^n)^m \times X$. These forms generate the space of all test forms.  If $\Phi_1$ is not of full bi-degree $\big(n(m-1), n(m-1)\big)$ then $\Phi_2$ has bi-degree strictly bigger than $(n-m-p+2,n-m-p+2)$, so  $ \langle(A_\lambda)_* \big( \widetilde{u}_1R_m\big),\Phi \rangle= 0$ for every $\lambda$ by degree reasons. Moreover, also by degree reasons,  the right hand side of (\ref{ine_nuinfty})  vanishes when paired with $\Phi$, so (\ref{ine_nuinfty})   implies that $\langle(A_\lambda)_*\big(\widetilde{u}_1 \ddc \widetilde{u}_2 \wedge \cdots \wedge \ddc \widetilde{u}_{m-1} \wedge \widetilde T \big), \Phi\rangle$ converges to $0$ as $|\lambda| \to \infty$. So (\ref{eq_limitAlmabdaddcu1loaitru}) holds when pairing with these forms.

Suppose now that $\Phi_1$ is of full bi-degree $\big(n(m-1), n(m-1)\big)$). Then  $\ddc \Phi = \ddc_{y^m} \Phi$, so
 $$\big \langle \widetilde{u}_1 \ddc \widetilde{u}_2 \wedge \cdots \wedge \ddc \widetilde{u}_{m-1} \wedge \widetilde T, \Phi \big \rangle = \big \langle \widetilde{u}_1 R_m, \Phi  \big\rangle,$$ which shows that  (\ref{eq_limitAlmabdaddcu1loaitru}) holds when pairing with $\Phi$. This proves  (\ref{eq_limitAlmabdaddcu1loaitru}).
 
Combining  (\ref{eq_limitAlmabdaddcu1loaitru}) and Lemma \ref{le_claim1} we get
\begin{equation} 
(A_\lambda)_* ( \widetilde u_1 \ddc \widetilde u_2\wedge \cdots \wedge \ddc \widetilde u_{m-1} \wedge \widetilde T)  \stackrel{|\lambda| \to \infty}{\longrightarrow} \pi^*(u_1 \ddc u_2\wedge \cdots \wedge \ddc u_{m-1} \wedge T),
\end{equation}
which gives (\ref{eq:density-convergence}) after taking $\ddc$ on both sides. This completes the proof.
\end{proof}

\begin{proposition} \label{pro_taunga} Let $R$ be a closed positive current on $X^m$ and let $\varrho$ be an admissible map from an open neighbourhood of $\Delta$ to the normal bundle $N\Delta$.  Assume that there is a sequence $\{\lambda_{k}\}_{k \in \N}$ of complex numbers converging to $\infty$ for which 
$$R_\infty:= \lim_{k \to \infty} (A_{\lambda_k})_* \varrho_* R$$
exists. Then $R_\infty$ is $\Delta$-conic, \emph{i.e,} $(A_\lambda)^* R_\infty = R_\infty$ for every $\lambda \in \C^*$.

Let $\Phi$  be a smooth form with compact support in $N \Delta$ and denote $\Phi_\lambda:= (A_\lambda)^* \Phi$. If $\tau$ is another admissible map then
\begin{align} \label{ine_taunaga}
\big|\langle R, \varrho^* \Phi_{\lambda_k} -  \tau^* \Phi_{\lambda_k} \rangle \big| \le C  |\lambda_k|^{-1}.
\end{align}
for some constant $C$ independent of $\lambda$. In particular,  $\lim_{k \to \infty} (A_{\lambda_k})_* \tau_* R$ also  exists and is equal to $R_\infty.$ 
\end{proposition}

\proof 
This proposition is implicitly proved in \cite{dinh-sibony:density}. We give here a proof for the readers' convenience. We emphasize that we don't assume $X$ to be K\"ahler nor the compactness of the intersection of $\supp R$ and $\Delta$.  In order to prove that $R_\infty$ is $\Delta$-conic, we use \cite[Le. 4.7]{dinh-sibony:density} to reduce the problem to the case where $X^m$ is replaced by its blow up along $\Delta$. Then the remainder is done as in the proof of \cite[The. 4.6]{dinh-sibony:density}.

We now prove the second statement. Let $(y', y^m)$ be local coordinates on $X^m$ for which $\Delta= \{y'=0\}.$  Since both $\tau$ and $\varrho$ are admissible there are matrices  $a(y)$ and $b(y)$ with smooth coefficients such that in these coordinates
$$\tau(y',y^m)= \big(y', y^m + y' a(y)\big)+ O(|y'|^2) \text{ and } \varrho(y',y^m)= \big(y', y^m + y' b(y)\big)+ O(|y'|^2).$$
In particular 
\begin{equation} \label{eq:rho-tau}
\left( \varrho - \tau \right) (y',y^m) = \big(0, y' c(y)\big)+ O(|y'|^2),
\end{equation}
where $c(y)=b(y)-a(y)$. Let $K_\lambda$ be the support of $\varrho^* \Phi_\lambda -  \tau^* \Phi_\lambda$. Notice that $|y'| \lesssim |\lambda|^{-1}$ on $K_\lambda$. Using this fact and (\ref{eq:rho-tau}) we can see that $$\varrho^* \Phi_\lambda -  \tau^* \Phi_\lambda = \left( A_\lambda \circ ( \varrho - \tau ) \right)^* \Phi = \frac{1}{|\lambda|} \varrho^* A_\lambda^* \Psi_\lambda $$ where $\Psi_\lambda$ are compactly supported forms whose coefficients are uniformly bounded in $\lambda$. Let $\Omega$ be a positive form with compact support such that $\Psi_\lambda \leq \Omega$ for every $\lambda$. We have then $$
\big|\langle  R, \varrho^* \Phi_\lambda -  \tau^* \Phi_\lambda \rangle \big| =  |\lambda|^{-1} \big|\langle (A_\lambda)_* \varrho_*  R, \Psi_\lambda \rangle \big| \leq |\lambda|^{-1}\big|\langle (A_\lambda)_* \varrho_*  R, \Omega \rangle \big|.$$
Since the limit of $(A_{\lambda_k})_* \varrho _* R$ as $\lambda_k \to \infty$ exists, their masses on compact sets are bounded uniformly in $\lambda_k$. This gives (\ref{ine_taunaga}) and finishes the proof.
\endproof

\begin{proof}[End of the proof of  Theorem \ref{thm:compatible-def-wedge-general}] 
Let $\varrho$ be the coordinates given by (\ref{eq:standard-tau}). Proposition  \ref{pro_voitauthoi} gives $$\lim_{\lambda \to \infty} (A_\lambda)_* \varrho_* (T_1 \otimes \cdots \otimes T_m) = \pi^*(T_1 \wedge \cdots \wedge T_{m-1} \wedge T)$$ and Proposition   \ref{pro_taunga} shows that
$$\lim_{\lambda \to \infty} (A_\lambda)_* \tau_* (T_1 \otimes \cdots \otimes T_m) = \pi^*(T_1 \wedge \cdots \wedge T_{m-1} \wedge T)$$
for any admissible map $\tau$, so by the definition of the Dinh-Sibony product we have $T_1 \wedge \cdots \wedge T_{m-1} \wedge T = T_1 \curlywedge \cdots \curlywedge T_{m-1} \curlywedge T$, concluding the proof.
\end{proof}


\section{Densities and the non-pluripolar product} \label{sec_nonpluri}

Let $X$ be a Kähler manifold and let $u_1,\ldots,u_m$ be p.s.h.\ functions on $X$. Let 
$$    \mathcal O_k := \bigcap_{j=1}^m \{ x \in X : u_j(x) > -k \}.$$

Following \cite{BT_fine_87,BEGZ} we say that the non-pluripolar product of $\ddc u_j$ $j=1,\ldots m$, denoted by $\langle \ddc u_1 \wedge \cdots \wedge \ddc u_m \rangle$ is well defined if the following limit exists
\begin{equation} \label{eq:non-pp-def}
\left \langle \ddc u_1 \wedge \cdots \wedge \ddc u_m \right \rangle = \lim_{k \to \infty} \mathbf 1_{\mathcal O_k} \bigwedge_{j=1}^m \left( \ddc \max \{u_j,-k\} \right),
\end{equation}
where the product on the right hand side is defined as in (\ref{eq:BT-def}). Equivalently,  $\langle \ddc u_1 \wedge \cdots \wedge \ddc u_m \rangle$  is well defined if
\begin{equation} \label{eq:non-pp-def-equiv}
  \sup_{k \in \N} \int_{\mathcal O_k \cap K} \bigwedge_{j=1}^m  \ddc \max \{u_j,-k\} \wedge \omega^{n-m} < +\infty
\end{equation}
for every compact $K \subset X$. Here $\omega$ stands for a fixed Kähler form.

Let now $T_1,\ldots,T_m$ be positive closed $(1,1)-$currents on $X$. Using local potentials and (\ref{eq:non-pp-def}) we can define their non-pluripolar product $\langle T_1 \wedge \cdots  \wedge T_m \rangle$.

If $u_j$ is a local potential of $T_j$, the pole set of $T_j$ $$I_{T_j} = \{x \in X : u_j(x) = -\infty\}$$ is well defined and independent of the chosen potential.

Suppose now that the supports of the $T_j$ have compact intersection. Then, following Section \ref{sec1}, we can define density currents for the collection $(T_1,\ldots,T_m)$, which are currents on the bundle $\pi: N\Delta\to \Delta \subset X^m$.  The following result compare these two notions.

\begin{theorem} \label{th_density_nonpluri}   Let $T_1, \cdots, T_m$ be positive closed $(1,1)-$currents. Assume their supports have compact intersection.  Then  the non-pluripolar product $ \langle T_1 \wedge \cdots \wedge T_m \rangle$ is well-defined on $X$ and  
\begin{align} \label{ine_TjS}
\pi^* \langle T_1 \wedge \cdots \wedge T_m \rangle \le   \mathbf{1}_{\pi^{-1}(X \setminus (\bigcup_{1 \le j \le m} I_{T_j}))} S,
\end{align}
for any density current $S$ associated with $(T_1, \ldots, T_m)$.

Assume furthermore that the sets $\mathcal O_k$ are open in the Euclidean topology. Then inequality \emph{(\ref{ine_TjS})} becomes an equality. 
\end{theorem}

\begin{proof}
Let $S$ be a density current associated with $(T_1,\ldots,T_m)$ and denote $$T:= T_1 \otimes \cdots \otimes T_m.$$
Let $K$ denote the intersection of $\supp T$ with $\Delta$.  Since we are assuming that the the $T_j$'s have compact intersection, $K$ is also compact. By \cite[Th. 4.6]{dinh-sibony:density}, $S$ is supported by $\pi^{-1}(K)$ and there is a constant $C$ for which
\begin{align} \label{ine_mass_density}
\|S\| \le C \|T\|_K,
\end{align}

Notice that the result we are aiming to prove is local, so from now we assume that $X$ is an open subset in $\C^n$. Notice that the supports of the $T_j$ no longer  have compact intersection in $X$, but this won't affect the proof.


 Write $T_j = \ddc u_j$ where the $u_j$ are p.s.h.\ in $X$. For  $k \in \N$ and $1 \le j \le m$ define $u_j^k:= \max\{u_j, -k\}$ and  $T_j(k):= \ddc u^k_j$.  Set 
$$ \quad T(k):= T_1(k) \otimes \cdots \otimes T_m(k) \text{ and } \,\, \widetilde{T}(k):= T_1(k) \wedge \cdots \wedge T_m(k).$$
Note that $\widetilde{T}(k)$ is a positive closed current on $X,$ whereas $T$ and $T(k)$ are positive closed currents on $X^m$ and $T(k)$ converges to $T$ as $k \rightarrow \infty.$  

Fix $\tau$ an admissible coordinate system. Since $S$ is a density current, there exists a sequence $\lambda_k \nearrow \infty$ such that $(A_{\lambda_k})_* \tau_* T$ converge to $S$. To simplify the notation, we can assume that   
\begin{align} \label{eq:AlambdaS}
\lim_{\lambda \to \infty}(A_{\lambda})_* \tau_* T = S.
\end{align}
This assumption does not affect the proof.

From the above convergence we get that for any open set $U \subset X^m$ and any closed set $K \subset X^m$ we have
\begin{equation} \label{eq:open-compact}
\liminf_{\lambda \to \infty} \mathbf{1}_U (A_{\lambda})_* \tau_* T \geq \mathbf{1}_U S \,\,\, \text{ and } \,\,\,  \limsup_{\lambda \to \infty} \mathbf{1}_K (A_{\lambda})_* \tau_* T \leq \mathbf{1}_K S.
\end{equation}

By Theorem \ref{thm:compatible-def-wedge-general} we have, for any fixed $k$, 
\begin{align} \label{eq_Alambda_k}
\lim_{\lambda \rightarrow \infty} (A_{\lambda})_*\big( \tau_* T(k)\big)= \pi_m^* \widetilde{T}(k).
\end{align}

For $k \in \N \cup \{\infty\}$, put 
$$\Omega_k:= \bigcap_{j=1}^m \{(x^1, \cdots, x^m) \in X^m:  u_j(x^j) > -k\}.$$


\begin{lemma} \label{lemma:3-ineq}
\begin{align} \label{limit_Ok}
\mathbf{1}_{\pi^{-1}(\Omega_k \cap \Delta)} \pi^* \big(\widetilde{T}(k) \big) &\le  \liminf_{\lambda \rightarrow \infty} (A_{\lambda})_*\tau_*\big(  \mathbf{1}_{\Omega_k} T(k)\big)\\
\nonumber
&\le \liminf_{\lambda \rightarrow \infty} (A_{\lambda})_* \tau_*\big(  \mathbf{1}_{\overline \Omega_k} T\big) \le   \mathbf{1}_{\pi^{-1}(\overline \Omega_k \cap \Delta)} S.
\end{align}
for every $k \in \N$.
\end{lemma}

Assume for now the Lemma is true. Its proof will be given below. Then (\ref{limit_Ok}) together with (\ref{ine_mass_density}) implies (\ref{eq:non-pp-def-equiv}), showing that $\langle T_1,\ldots T_m \rangle$ is well defined.  By letting $k \to \infty$ in  (\ref{limit_Ok}), using the definition of non-pluripolar product and noticing that $ \bigcup_{k \geq 1} \overline \Omega_k \cap \Delta =  X \setminus \bigcup_{1 \le j \le m} I_{T_j}$ we get inequality (\ref{ine_TjS}).

Suppose now that the $\mathcal O_k$ are open. Then  the $ \Omega_k$ are also open. Using the same type of argument as in the proof of Lemma \ref{lemma:3-ineq} one gets 
\begin{align} \label{ine_limsupATk}
\limsup_{\lambda \rightarrow \infty} (A_{\lambda})_* \tau_*\big(  \mathbf{1}_{\Omega_{k}}T\big)= \limsup_{\lambda \rightarrow \infty} (A_{\lambda})_* \tau_* \big(  \mathbf{1}_{\Omega_k}T(k)\big) \le \mathbf{1}_{\pi^{-1}(\overline{\Omega}_k \cap \Delta)} \pi^* \big(\tilde{T}(k) \big).
\end{align}

Since $\Omega_k$ is open, (\ref{ine_limcotynua}) holds. Using (\ref{eq:open-compact}) this gives
$$\limsup_{\lambda \rightarrow \infty} (A_{\lambda})_* \tau_*\big( \mathbf{1}_{\Omega_{k}} T\big) \ge \mathbf{1}_{\pi^{-1}(\Omega_k \cap \Delta)} S,$$
which together with (\ref{ine_limsupATk}) yields
$$\mathbf{1}_{\pi^{-1}(\overline{\Omega}_{k} \cap \Delta)} \pi^* \big(\tilde{T}(k) \big) \ge \mathbf{1}_{\pi^{-1}(\Omega_k \cap \Delta)} S.$$

Letting $k \to \infty$ gives
\begin{align} \label{ine_Oinfty2}
\langle T_1, \cdots, T_m \rangle \ge \mathbf{1}_{\pi^{-1}( X \setminus \bigcup_{1 \le j \le m} I_{T_j})} S,
\end{align}
giving the equality in (\ref{ine_TjS}).

\begin{proof}[Proof of Lemma \ref{lemma:3-ineq}. ]
Let us first prove the first inequality in (\ref{limit_Ok}).  If each  $u_j$ is continuous, then $\Omega_k$ is open and 
\begin{align} \label{ine_limcotynua} 
 \lim_{\lambda \rightarrow \infty} (A_{\lambda})_* \mathbf{1}_{\tau(\Omega_k)}= \mathbf{1}_{\pi^{-1}(\Omega_k \cap \Delta)}.
 \end{align}
Then the first inequality follows from (\ref{eq:open-compact}).

Consider now the general case. Fix a small positive constant $\epsilon$. Consider $u_j$ as a psh function on $X^m.$  By the quasi-continuity of psh functions (see \cite{klimek}), there exists a closed subset $V$ of $X^m$ such that $u_j$ is continuous on $V$ and 
 \begin{equation} \label{ine_cap}
\capK(X^m \backslash V,X^m) <\epsilon.
 \end{equation} 
As the $u_j$ are also psh on $\Delta$ we can choose $V$ such that
\begin{equation}
 \label{ine_VDelta}
\capK(\Delta \backslash (V \cap \Delta),\Delta) \le \epsilon
\end{equation}

Since $\tau$ and $A_\lambda$ are holomorphic isomorphisms we also have 
\begin{align} \label{ine_cap2}
\capK\big(A_\lambda \circ \tau(X^m \backslash V), N \Delta \big) < \epsilon. 
 \end{align}
 
Let $f_j$, $j=1,\ldots,m$ be continuous functions on $X^m$ such that $f_j= u_j$ on $V$. In particular we have that  
$$\{f_j > -k\} \triangle \{ u_j > -k\} \subset X^m \backslash V.$$
The last fact combined with (\ref{ine_cap2})  implies that
\begin{align} \label{sosanhfjTk}
\| (A_{\lambda})_*  \tau_*\big[ \mathbf{1}_{\Omega_k}T(k) -  \mathbf{1}_{\Omega_k'}  T(k) \big] \| = O(\epsilon)
\end{align}
independently of $\lambda$, where 
$$\Omega_k':= \bigcap_{j=1}^m\{f_j >-k\}.$$

Since  $\Omega_k'$ is open we have
\begin{align} \label{limit_O'}
\lim_{\lambda \rightarrow \infty} (A_{\lambda})_* \mathbf{1}_{\Omega_k'} = \mathbf{1}_{\pi^{-1}(\Omega_k' \cap \Delta)} 
\end{align}
which together with (\ref{eq:open-compact}) implies that
\begin{align} \label{conve_fj}
\liminf_{\lambda \rightarrow \infty} (A_{\lambda})_* \tau_* \big(  \mathbf{1}_{\Omega_k'} T(k)\big) \ge \mathbf{1}_{\pi^{-1}(\Omega_k' \cap \Delta)} \pi^* \big(\tilde{T}(k) \big).
\end{align}

Consider now the difference
$$ \mathbf{1}_{\pi^{-1}(\Omega_k' \cap \Delta)} \pi^* \big(\tilde{T}(k) \big)-  \mathbf{1}_{\pi^{-1}(\Omega_k \cap \Delta)} \pi^* \big(\tilde{T}(k) \big).$$ 
Its mass is at most  two times that of  $ \mathbf{1}_{\pi^{-1}\big(\Delta \backslash (V \cap \Delta)\big)} \pi^* \big(\tilde{T}(k) \big)$ which is 
$$\lesssim \capK \big(\Delta \backslash (V  \cap \Delta), \Delta \big) \le \epsilon $$
by (\ref{ine_VDelta}).  Combining this with (\ref{conve_fj}) and (\ref{sosanhfjTk})and making $\epsilon \to 0$ gives the first inequality of (\ref{limit_Ok}).

For the second inequality we observe that $\mathbf{1}_{\Omega_k} T(k) = \mathbf{1}_{\Omega_k} T(\ell) \leq \mathbf{1}_{\overline \Omega_k} T(\ell) $ for $\ell \geq k$. This follows from \cite[Corollary 4.3]{BT_fine_87}. Taking the $\limsup$ with respect to $\ell$ and remembering that $T(\ell) \stackrel{\ell \to \infty}{\longrightarrow} T$ weakly  gives $\mathbf{1}_{\Omega_k} T(k) \leq \mathbf{1}_{\overline \Omega_k} T $.

To prove the third inequality of (\ref{limit_Ok}) it suffices to observe that
\begin{align} \label{ine_chantren}
\limsup_{\lambda \rightarrow \infty} (A_{\lambda})_*\tau_*  \mathbf{1}_{\overline \Omega_k} \le \mathbf{1}_{\pi^{-1}(\overline{\Omega}_k \cap \Delta)}.
\end{align} 
and use (\ref{eq:AlambdaS}) and (\ref{eq:open-compact}). We have thus proven (\ref{limit_Ok}).
\end{proof}

The proof of Theorem \ref{th_density_nonpluri} is now complete.
\end{proof}


Recall that a current $T$ is said to have analytic singularities if locally $T= \ddc u$ where
\begin{equation} \label{eq:an-sing}
u = \frac{c}{2} \log \sum_{j=1}^n |f_j|^2 + v,
\end{equation}

where $c > 0$, $v$ is smooth and $f_j$ are local holomorphic functions. In that case, the pole set of $T$ is an analytic set $Z$ called the singular locus of $T$ and $T$ is smooth outside $Z$. Notice that in this case $\mathcal O_k = \{u > -k\}$ is open for every $k$, so we may then apply the second part  of Theorem \ref{th_density_nonpluri}.

\begin{corollary} \label{cor:non-pp-analytic-sing}
Let $T_1, \cdots, T_m$ be positive closed $(1,1)-$currents with analytic singularities on a complex manifold $X$ whose supports have compact intersection. Denote by $Z_j$ the singular locus of $T_j$ and let $Z = \bigcup_{j=1}^m Z_j$. Then  the non-pluripolar product $ \langle T_1 \wedge \cdots \wedge T_m \rangle$ is well-defined $X$ and  
\begin{equation*}
\pi^* \langle T_1 \wedge \cdots \wedge T_m \rangle =   \mathbf{1}_{\pi^{-1}(X \setminus Z)} S,
\end{equation*}
for any density current $S$ associated with $(T_1, \ldots, T_m)$.
\end{corollary}

\section{Non-proper intersection}  \label{sec_nonproper}

In this section we  first study some particular cases where the density product is not well-defined but a description of the density current can be explicitly given. We also give a comparison of density currents under  blowups. As a consequence, Theorem \ref{th_AWdensity} will follow immediately from these results. 

\subsection{Currents with divisorial singularities}
Let $X$ be an arbitrary complex manifold of dimension $n$.  A positive closed $(1,1)-$current $T$ on $X$ is said to have \textit{divisorial singularities} if locally we can write
\begin{equation} \label{eq:div-sing}
T = \ddc \log |f| + \ddc v = [f=0] + \ddc v, 
\end{equation}
where  $f$ is a holomorphic function and $v$ is continuous. This is a special case of a current with analytic singularities. Notice that $v$ is psh outside $Z = \{f=0\}$ and so, up to modifying its value on  $Z$, we may assume that $v$ is psh.

Following \cite{dinh-sibony:density} we say that a  a positive closed $(m,m)-$current $S$ in $\pi: \overline{N \Delta} \to \Delta$ with $m \leq n$  is of \textit{maximal $h-$dimension} if $S \wedge \pi^* \Phi \neq 0$ for any smooth positive $(n,n)$-form $\Phi$ on $\Delta$.

\begin{theorem} \label{thm:densityTm}
Let $T$ be as in \emph{(\ref{eq:div-sing})} and suppose that $Z = \{f=0\}$ is smooth. Then for $m=1,\ldots,n$ the collection $(T, \ldots, T)$ of $m$ copies of $T$ admits a unique density current which is given locally by
\begin{equation} \label{eq:densityTm}
T^{\otimes m}_\infty = \pi^* \left((\ddc v)^m + m \, (\ddc v)^{m-1} \wedge  [f=0] \right) + \sum_{k=2}^m  \ell^{k-1} \cdot \pi^*( \,[f=0] \wedge (\ddc v)^{m-k}) \wedge R_k,
\end{equation}
where $\ell$ is a positive integer and $R_k$ is $(k-1,k-1)-$current is of maximal $h-$dimension which is positive, closed and whose restriction to each fiber of $\overline{N\Delta}$ is cohomologous to $\binom{m}{k}$ times the class of a linear subspace on that fiber.
\end{theorem}

\begin{proof} We may assume that $X$ is a domain in $\C^n$ and use the the admissible map  $\varrho$  as in  (\ref{eq:standard-tau}). By Proposition \ref{pro_taunga}, we only need to prove that $\lim_{\lambda \to \infty} (A_\lambda)_* \varrho_* T^{\otimes m}$ equals the right-hand side of (\ref{eq:densityTm}).  Suppose at first that the divisor $[f=0]$ is reduced. The non reduced case is treated in the end of the proof. Since $Z$ is smooth by hypothesis, this implies that $Df$ is non vanishing on $Z$.

Let $T_1 =  [f=0]$ and $T_2 = \ddc v$, so that $T = T_1 + T_2$.  Observe that the products $T_1 \wedge T_2^{m-1}$ and $T_2^m$ are classically defined and no other product of degree $m$ involving $T_1$ and $T_2$ is well defined.

From Theorem \ref{thm:compatible-def-wedge-general}, the term $T_2^{\otimes m}$  in the tensor product $T^{\otimes m} = (T_1+T_2)^{\otimes m}$ has $\pi^*(\ddc v)^m$ as a unique density current and the $m$ terms containing one copy of $T_1$ and $m-1$ copies of $T_2$ have $\pi^* \left( (\ddc v)^{m-1} \wedge  [f=0] \right)$ as a unique density current. This gives the first term on the right hand side of (\ref{eq:densityTm}).

Consider now a general term of $T^{\otimes m}$ containing $k$ copies of $T_1$ and $m-k$ copies of $T_2$. Up to a permutation of the factors, we may assume that this terms is $T_2^{\otimes m-k} \otimes T_1^{\otimes k}$. Put 
$$S_\lambda := (A_\lambda)_* \varrho_* (T_2^{\otimes m-k} \otimes T_1^{\otimes k}).$$
We have
\begin{equation*}
\begin{split}
S_\lambda &=  \ddc v(\lambda^{-1} y^1 + y^m) \wedge \cdots \wedge\ddc v(\lambda^{-1} y^{m-k} + y^m) \wedge \\
& ~~  \wedge \ddc \log |f(\lambda^{-1} y^{m-k+1} + y^m)| \wedge \cdots \wedge \ddc \log |f(\lambda^{-1} y^{m-1} + y^m)| \wedge \ddc \log |f( y^m)|,
\end{split}
\end{equation*}
Let 
$$\widetilde Z = \{f(y^m)=0\}, \quad S'_\lambda = \ddc v(\lambda^{-1} y^1 + y^m) \wedge \cdots \wedge\ddc v(\lambda^{-1} y^{m-k} + y^m) $$
 and   
$$ S''_\lambda= \ddc \log |f(\lambda^{-1} y^{m-k+1} + y^m)| \wedge \cdots \wedge  \ddc \log |f(\lambda^{-1} y^{m-1} + y^m)|,$$
so that $S_\lambda =  S'_\lambda \wedge S''_\lambda \wedge [\widetilde Z ]$.

We can rewrite $$S''_\lambda = \ddc \log |\lambda f(\lambda^{-1} y^{m-k+1} + y^m)| \wedge \cdots \wedge  \ddc \log | \lambda f(\lambda^{-1} y^{m-1} + y^m)|.$$
Using the Taylor expansion of $f$ and restricting the potentials of $S''_\lambda$ to $\widetilde Z$ we have that $S''_\lambda \wedge [\widetilde Z] = U_\lambda \wedge [\widetilde Z]$ where
\begin{equation*}
U_\lambda = \ddc \log |Df(y^m) \cdot y^{m-k+1} + O(\lambda^{-1})| \wedge \ldots \wedge  \ddc \log |Df(y^m) \cdot y^{m-1} + O(\lambda^{-1})| 
\end{equation*}

Since $Df$ is non vanishing on $Z$, the intersection $$\{Df(y^m) \cdot y^{m-k-1}= 0\} \cap \cdots \cap \{Df(y^m) \cdot y^{m-1}= 0\} \cap \widetilde Z$$ is proper and by continuity the intersection defining $U_\lambda \wedge [\widetilde Z]$ is also proper for $\lambda$ large, so by  Lemma \ref{lemma:continuity-proper-int} below we have
$$\lim_{\lambda \to \infty} U_\lambda \wedge [\widetilde Z]  = U_k \wedge [\widetilde Z],$$
where
\begin{equation} \label{eq:Uk}
U_k := \ddc \log |Df(y^m) \cdot y^{m-k+1}| \wedge \ldots \wedge  \ddc \log |Df(y^m) \cdot y^{m-1}|
\end{equation}

Since $v$ is continuous we have that $v(\lambda^{-1}y^j + y^m) \ \to v(y^m) = (v \circ \pi)(y^m)$ uniformly on compact sets, hence 
\begin{equation*}
\lim _{\lambda \to \infty} S_\lambda =  \lim _{\lambda \to \infty} \pi^*(\ddc v)^{m-k} \wedge U_\lambda \wedge [\widetilde Z] =  \pi^* \left((\ddc v)^{m-k} \wedge [f=0]\right) \wedge U_k.
\end{equation*}

Gathering all $\binom{m}{k}$ terms containing $m$ copies of $T_1$ and $m-k$ copies of $T_2$ we get a density current of the form $\pi^* \left((\ddc v)^{m-k} \wedge [f=0]\right) \wedge R_k,$ where $R_k$ is a sum of $\binom{m}{k}$ currents of the type (\ref{eq:Uk}) up to a permutation of coordinates. This gives the second term on the right hand side of (\ref{eq:densityTm}). Here $\ell = 1$. Notice that $U_k$ is of maximal $h-$dimension and its restriction to a fiber of $\overline{N\Delta}$ is a current of integration along a linear subspace. Therefore $R_k$ is of maximal $h-$dimension and its restriction to a fiber of $\overline{N\Delta}$ is cohomologous to $\binom{m}{k}$ times the  class of a linear subspace. This finishes the proof in the case where $[f=0]$ is reduced. 

If the divisor $[f=0]$ is non reduced we can write locally $[f=0] = \ell \cdot [g=0]$ for some integer $\ell$ and a  holomorphic function $g$ whose associated divisor is reduced. If we write $T = [f=0]  + \ddc v = \ell ( [g=0] + \ddc (\ell^{-1} v) )$ we get $T^{\otimes m} = \ell^m ([g=0] + \ddc (\ell^{-1}v))^{\otimes m}$ . Applying the formula in the reduced case to $([g=0] + \ddc (\ell^{-1}v))^{\otimes m}$ gives (\ref{eq:densityTm}).
\end{proof}

\begin{remark} \label{rmk:non-pp-aw}
If $T$ is as above, then its non-pluripolar self-product is $$\langle T^m \rangle = (\ddc v)^m,$$ so $\pi^* \langle T^m \rangle $ is a component of $T^{\otimes m}_\infty$. Notice that this is compatible with Corollary \ref{cor:non-pp-analytic-sing}.

Another way of computing the self-product of $T$ was proposed by Andersson-Wulcan, see \cite{andersson-wulcan}. Their product is defined recursively as $T^k_{AW} = \ddc \left( u \, \mathbf {1}_{X\setminus \{f=0\}} T^{k-1}_{AW}\right)$, where $u =  \log |f| + v$. This gives $T^m_{AW} = (\ddc v)^m + (\ddc v)^{m-1} \wedge [f=0]$, which is, up to the constant $m$ multiplying $(\ddc v)^{m-1} \wedge  [f=0]$,  the vertical component appearing in (\ref{eq:densityTm}). In Section \ref{sec:AW} below we use this fact in order to compare density currents and the Andersson-Wulcan product for a general current with analytic singularities.
\end{remark}

\begin{remark}
If we don't assume $Z$ to be smooth the situation becomes more complicated and it is not clear to us what the general formula should be. One issue that appears is that, even if the divisor of $f$ is reduced, the intersection in (\ref{eq:Uk}) is no longer well defined if $Z$ is singular.  
\end{remark}

The following Lemma used above is well known. See for instance \cite{chirka}, $\S 12.3$ and $\S 16.1$.

\begin{lemma}[Continuity of proper intersections] \label{lemma:continuity-proper-int}  Let $A_1^n,\ldots,A_k^n$, $n=1,2,\ldots $ be pure-dimensional analytic subsets of a complex manifold $X$ converging to $A_1,\ldots, A_k$ as $n\to \infty$. Suppose that $A_1^n,\ldots,A_k^n$ intersect properly for every $n$ and that $A_1,\ldots, A_k$ also intersect properly. Then
$$\lim_{n\to \infty} [A_1^n] \wedge \cdots \wedge [A_k^n] = [A_1] \wedge \cdots \wedge [A_k]$$ as currents in $X$.
\end{lemma}

\subsection{Simple normal crossings divisor}
Similar computations as the ones of the former section allow us to compute the self-density of a divisor with simple normal crossings (SNC) singularities. 

Let $D$ be a SNC divisor on $X$. By definition,
\begin{equation}
D= \alpha_1 D_1 + \cdots +\alpha_r D_r,
\end{equation}
where $\alpha_i \geq 0$, the $D_j$ are smooth irreducible hypersurfaces and around any point of $\bigcup D_j$ we can find a coordinate system $(z_1,\ldots,z_n)$ such that $D_j = \{z_j=0\}$ for $1\leq j \leq r$. In particular $r \leq n$.

In order to  compute the self density of order $m$ of the $(1,1)$ current $[D]$, we first expand the tensor product $\left( \sum \alpha_j [D_j] \right)^{\otimes m}$. We are then left with the task of computing the density of mixed tuples $([D_{j_1}] , \ldots , [D_{j_m}])$.

\begin{proposition}
The collection $([D_{j_1}] , \ldots , [D_{j_m}])$ admits a unique density current, which is given by $$T_{j_1,\ldots,j_m}^\infty = \pi^* \left( \bigwedge_{i \in I} [D_i] \right) \wedge R,$$
where the indices in $I$ run over the distinct $j_1,\ldots,j_m$ and $R$ is a current of maximal $h-$dimension.
\end{proposition}

\begin{proof}
If $(z_1,\ldots,z_n)$ is a local coordinate system as above and   $(z^1,\ldots,z^m)$ are the canonical coordinates on $X^m$ then, the current  $T_{j_1,\ldots,j_m}$ is given by
$$[z^{1}_{j_1} = 0] \wedge \cdots \wedge [z^{m}_{j_m} = 0].$$

Consider the collection of admissible coordinates $$\tau_k(z) = (z^1-z^k, \ldots, z^{k-1} - z^k,z^k,z^{k+1} - z^k, \ldots, z^{m} - z^k).$$

Set $r_1 = j_1$ and let $T_1$ be the intersection of all $[z^{k}_{j_k} = 0]$ with $j_k =r_1 = j_1$. Arguing as in the proof of Theorem \ref{thm:densityTm} we see that $(A_\lambda)_*(\tau_1)_* T_1$ converges to $\pi^*[D_{j_1}] \wedge R_1$ where $R_1$ is a current of maximal $h-$dimension. Furthermore, by Proposition \ref{pro_taunga}, we have that $(A_\lambda)_*\tau_* T_1$ converge to $\pi^*[D_{j_1}] \wedge R_1$ as $\lambda \to \infty$ for any other admissible map $\tau$.

Let $r_2$ be the first index among the $j_k$'s such that $j_{k} \neq r_1$ and let $T_2$  be the intersection of all $[z^{k}_{j_k} = 0]$ with $j_k = r_2$. Using the coordinates $\tau_{r_2}$ and arguing as above we see that  $(A_\lambda)_*(\tau_{r_2})_* T_2$ converge to $\pi^*[D_{r_2}] \wedge R_2$.

Continuing this procedure we get distinct numbers $r_1,\ldots,r_\ell$ and currents $T_1,\ldots, T_\ell$ such that $T_{j_1,\ldots,j_m} = T_1 \wedge \cdots \wedge T_\ell$ and that for every $j=1,\ldots,\ell$ and any admissible map $\tau$, the currents $(A_\lambda)_*\tau_* T_j$  converge to  $\pi^*[D_{r_j}] \wedge R_j$ where $R_j$ is of maximal $h-$dimension.

The currents $R_j$ intersect properly so using Lemma \ref{lemma:continuity-proper-int} we see that the unique density current associated with  $([D_{j_1}] , \ldots , [D_{j_m}])$  is given by
$$T_{j_1,\ldots,j_m}^\infty = \pi^* \left( [D_{r_1}] \wedge \cdots \wedge [D_{r_\ell}] \right) \wedge R_1 \wedge \cdots R_\ell,$$
which completes the proof.
\end{proof}

From the above proposition, it is straightforward to derive a formula for the density of $([D],\ldots,[D])$. In order to avoid cumbersome notation, we leave it to the reader the task of writing down the explicit formula.

\begin{corollary} \label{cor:densityTm2} Let $T$ be a positive closed $(1,1)$-current with divisorial singularities. Assume that singular locus of $T$ is a simple normal crossing divisor. Then the limit $\lim_{\lambda \to \infty} (A_\lambda)_* \tau_* T^{\otimes m}$ exists for every admissible map $\tau.$ In particular, there is a unique density current associated with the $m$-tuple $(T, \ldots, T)$.
\end{corollary}

\subsection{Comparison with the Andersson-Wulcan product} \label{sec:AW}
We now use the above results to compare density currents and the Andersson-Wulcan product for a general current with analytic singularities.

Let $X$ be a complex manifold and let $T$ be a positive closed $(1,1)-$current on $X$ with analytic singularities, see (\ref{eq:an-sing}). Denote by $Z$ the singular locus of $T$. The Andersson-Wulcan self-product is defined recursively by $$T^k_{AW} = \ddc \left( u \, \mathbf {1}_{X\setminus Z} T^{k-1}_{AW}\right),$$ where $u$ is a local potential of $T$, see \cite{andersson-wulcan}. It coincides with the classical product for $k \leq \codim Z$.

We will work with a resolution of singularities in order to reduce our problem to the divisorial case. So let $h: X' \to X$ be a proper surjective holomorphic map between two complex manifolds such that $p$ is an isomorphism outside an analytic subset $E'$ of $X'$ and such that $E:=p(E')$ is of codimension at least $2$. A typical example of such $h$ is the blow-up along a submanifold of codimension $\ge 2$ of $X.$  

Let $T_1, \ldots, T_m$ be positive closed $(1,1)$-currents on $X$.  The pull-back $h^* T_j$ can be defined by pulling back the local potentials of $T_j$. Define $\tilde{h}: (X')^m\to X^m$ by putting 
$$\tilde{h}(x^1, \cdots, x^m):= \big(h(x^1), \ldots, h(x^m)\big).$$
Since $\tilde{h}$ sends the diagonal $\Delta_{X'}$ of $(X')^m$ to the diagonal $\Delta_X$ of $X^m,$ the differential $D \widetilde{h}$ of $\widetilde h$ induces a well-defined bundle map from  $N \Delta_{X'}$ to $N \Delta_X$ which is also denoted by $D \widetilde{h}$ for simplicity. 


 
\begin{proposition}  \label{pro.AW} Assume the collection $(h^*T_1, \ldots, h^*T_m)$ admits a density current $S'$ with defining sequence $\{\lambda_k\}_{k\in\N}$. If $S$ is a density current associated with $(T_1, \ldots, T_m)$ defined by the same sequence $\{\lambda_k\}_{k\in\N}$ then 
\begin{align} \label{ine_hsaoT}
(D\tilde{h})_* S' \le   S. 
\end{align}
\end{proposition}

\proof  Notice that $D\tilde{h}$ is not proper on $N\Delta_{X'}$, so we need to check that $(D\tilde{h})_* S'$ is well-defined. Since $D\tilde{h}$ is linear on fibers, it can be extended as a map from  $\overline{N\Delta}_{X'}$ to  $\overline{N\Delta}_{X}.$ On the other hand,  by Proposition \ref{pro_taunga}, $S'$ is $\Delta_{X'}$-conic thus it can be extended to a current in $\overline{N\Delta}_{X'}.$ As a consequence, the pushforward  $(D\tilde{h})_* S'$ is well-defined as a current in  $\overline{N\Delta}_{X}$ which can be restricted to $N\Delta_{X'}$.

Because the problem is local in $X$ we can assume that $X$ is an open subset of $\C^n$.  We can use local coordinates  $\varrho$  as in (\ref{eq:standard-tau}) in $X^m$ and identify $\Delta$  with $\{y'=0\}$ and $N \Delta_X$ with $(\C^n)^{m-1} \times X$. Let $\tau'$ be an admissible map defined in a neighborhood of $\Delta'$.

By assumption we have
$$S= \lim_{k \to \infty} (A_{\lambda_k})_* \varrho_*(T_1 \otimes \cdots \otimes T_m) \text{ and } S' = \lim_{k \to \infty} (A_{\lambda_k})_* \tau'_* (h^*T_1 \otimes \cdots \otimes h^*T_m).$$

Let $0\le  \chi \le 1$ be a smooth function compactly supported on $N \Delta_{X'}$ and set $\chi_\lambda:= \chi \circ A_\lambda.$ Then
$$\lim_{k \to \infty} (A_{\lambda_k})_* \big(\chi_{\lambda_k}\tau'_* (h^*T_1 \otimes \cdots \otimes h^*T_m)\big)= \chi S'.$$

Combining this with the fact that $A_\lambda$ and $D \tilde{h}$ commute gives
\begin{align} \label{eq_DDhnga}
(D\tilde{h})_* (\chi S') = \lim_{k \to \infty} (A_{\lambda_k})_* (D \tilde{h})_*\big( \chi_{\lambda_k} \tau'_*\big( h^*T_1 \otimes \cdots \otimes h^*T_m\big).
\end{align}

We claim that
\begin{align} \label{eq_thayhnga}
(D\tilde{h})_* (\chi S') = \lim_{k \to \infty} (A_{\lambda_k})_* \varrho_* \big( \chi_{\lambda_k} \widetilde h_*\big( h^*T_1 \otimes \cdots \otimes h^*T_m\big).
\end{align}

For this purpose, let $$g:= D \widetilde h \circ \tau' - \varrho \circ \widetilde h: (X' )^m \to N \Delta_X \simeq (\C^n)^{m-1} \times X,$$ so that proving (\ref{eq_thayhnga}) is equivalent to  proving  that for  every test form $\Phi$ in $(\C^n)^{m-1} \times X,$
$$\lim_{k\to \infty}\big\langle  (A_{\lambda_k})_* g_* \big( \chi_{\lambda_k} \tau'_*(h^*T_1 \otimes \cdots \otimes h^*T_m) \big), \Phi \big\rangle \to 0$$
which is, in turn, equivalent to 
$$\lim_{k\to \infty}\big\langle   h^*T_1 \otimes \cdots \otimes h^*T_m , \Phi(\lambda_k) \big \rangle \to 0,$$
where $\Phi(\lambda):=\tau'^*\chi_\lambda g^* (A_\lambda)^* \Phi$.

Since $D \widetilde h$ is induced by the differential of $h$ and $\tau'$ and $\varrho$ are admissible we can see that, in any local holomorphic coordinate system  $z=(z',z'')$ on $(X')^m$ where $\Delta'$ is given by $\{z'=0\}$ we have $g(z)= \big(0, O(|z'|) \big) + O(|z'|^2)$, compare with (\ref{eq:rho-tau}).

Since  the support of $\chi_\lambda$ is contained in tubular neighborhood of $\Delta_{X'}$ of radius $\sim |\lambda|^{-1}$ we can argue as in the proof of Proposition \ref{pro_taunga} and write $\Phi(\lambda) = |\lambda|^{-1} \tau'^*(A_\lambda)^*( \Psi_\lambda)$, where the coefficients of $\Psi_\lambda$ are bounded uniformly in $\lambda$. 
This yields
\begin{align*}
\lim_{k\to \infty} \langle  h^*T_1 \otimes \cdots \otimes h^*T_m , \Phi(\lambda_k) \rangle = \lim_{k\to \infty} \frac{1}{|\lambda_k|} \langle  (A_{\lambda_k})_* \tau'_*   \big( h^*T_1 \otimes \cdots \otimes h^*T_m\big), \Psi_{\lambda_k} \rangle = 0,
\end{align*}
proving (\ref{eq_thayhnga}).

From the fact that $\codim E \geq 2$ we get
\begin{align} \label{eq_hsaotrensuoi}
h_* (h^* T_j)= T_j \text{ for } 1 \le j \le m.
\end{align}
By (\ref{eq_thayhnga}), the fact that $|\chi_\lambda| \le 1$ and (\ref{eq_hsaotrensuoi}) we see that   
\begin{align*}
 (D\tilde{h})_* (\chi S') & \le \lim_{k \to \infty} (A_{\lambda_k})_*  \varrho_* \widetilde h \big( h^*T_1 \otimes \cdots \otimes h^*T_m\big)\\
&= \lim_{k \to \infty} (A_{\lambda_k})_* \varrho_*(T_1 \otimes \cdots \otimes T_m)=S.
\end{align*}
Since $\chi$ is arbitrary, inequality (\ref{ine_hsaoT}) follows, finishing the proof. 
\endproof

\begin{remark} Consider  the case where $X'$ is compact and $T_1, \cdots, T_m$ are closed positive currents of higher bi-degree on $X.$ The strict transform $h^* T_j$ is still well-defined; see \cite{DinhSibony_pullback}. With the same proof, Proposition \ref{pro.AW} remains true if we assume furthermore that $T_j$ has no mass on $E$ for $1 \le j \le m.$  
\end{remark}


\proof[End of the proof of Theorem \ref{th_AWdensity}]   
By Hironaka's theorem we can find a modification  $h: X' \to X$ such that $h^* T$ is a current with divisorial singularities with simple normal crossings. Then we have, locally,  $h^* T= \ddc \log |f|+ \ddc v$, where $f$ is a holomorphic function and $v$ is psh and continuous.  By \cite[Eq. 4.5]{andersson-wulcan} we have that
\begin{equation} \label{eq:AW-pushforward}
T^m_{AW} = h_*\big( (\ddc v)^{m-1} \wedge [f=0] + (\ddc v)^m \big).
\end{equation}

Let $S'$ be the unique density current associated with $h^*T, \ldots, h^* T$ given by Corollary \ref{cor:densityTm2}. Computing the formula for $S'$ with the methods of the last section we see that $S'$ contains $\pi_{X'}^*(m \, (\ddc v)^{m-1} \wedge [f=0] + (\ddc v)^m )$ in its vertical component (compare with (\ref{eq:densityTm})). In particular 
$$S' \ge \pi_{X'}^*\big( (\ddc v)^{m-1} \wedge [f=0] + (\ddc v)^m  \big),$$

Let $S$ be a density current associated with $T, \ldots,T$. By Proposition \ref{pro.AW}, (\ref{eq:AW-pushforward}) and the fact that $ \pi_X \circ D \widetilde h = h \circ \pi_{X}$ we get 
\begin{align*}
S \ge (D\tilde{h})_* S' &\ge (D \tilde{h})_* \pi_{X'}^*\big( (\ddc v)^{m-1} \wedge [f=0] + (\ddc v)^m  \big) \\ &=  \pi_X^* h_*  \big( (\ddc v)^{m-1} \wedge [f=0] + (\ddc v)^m  \big) = \pi_X^*T^{m}_{AW}.
\end{align*}
\endproof

\begin{remark}
Notice that, for $m\leq \codim Z$, the inequality (\ref{eq:AW-ineq}) becomes an equality, since both sides coincide with the pullback by $\pi$ of the classical product.
\end{remark}

\bibliographystyle{alpha}
\bibliography{refs}

\begin{thebibliography}{BEGZ10}

\bibitem[AW14]{andersson-wulcan}
Mats Andersson and Elizabeth Wulcan.
\newblock Green functions, {S}egre numbers, and {K}ing's formula.
\newblock {\em Ann. Inst. Fourier (Grenoble)}, 64(6):2639--2657, 2014.

\bibitem[BEGZ10]{BEGZ}
S\'ebastien Boucksom, Philippe Eyssidieux, Vincent Guedj, and Ahmed Zeriahi.
\newblock Monge-{A}mp\`ere equations in big cohomology classes.
\newblock {\em Acta Math.}, 205(2):199--262, 2010.

\bibitem[BT76]{Bedford_Taylor_76}
Eric Bedford and B.~A. Taylor.
\newblock The {D}irichlet problem for a complex {M}onge-{A}mp\`ere equation.
\newblock {\em Invent. Math.}, 37, 1976.

\bibitem[BT87]{BT_fine_87}
Eric Bedford and B.~A. Taylor.
\newblock Fine topology, \v silov boundary, and {$(dd^c)^n$}.
\newblock {\em J. Funct. Anal.}, 72(2):225--251, 1987.

\bibitem[Chi89]{chirka}
E.~M. Chirka.
\newblock {\em Complex analytic sets}, volume~46 of {\em Mathematics and its
  Applications (Soviet Series)}.
\newblock Kluwer Academic Publishers Group, Dordrecht, 1989.
\newblock Translated from the Russian by R. A. M. Hoksbergen.

\bibitem[CLN69]{Chern_Levine_Nirenberg}
S.~S. Chern, Harold~I. Levine, and Louis Nirenberg.
\newblock Intrinsic norms on a complex manifold.
\newblock In {\em Global {A}nalysis ({P}apers in {H}onor of {K}. {K}odaira)},
  pages 119--139. Univ. Tokyo Press, Tokyo, 1969.

\bibitem[Dem]{demailly:agbook}
Jean-Pierre Demailly.
\newblock {\em Complex Analytic and Differential Geometry}.
\newblock \url{http://www-fourier.ujf-grenoble.fr/~demailly/}.

\bibitem[Dem93]{demailly:lelong}
Jean-Pierre Demailly.
\newblock Monge-{A}mp\`ere operators, {L}elong numbers and intersection theory.
\newblock pages 115--193, 1993.

\bibitem[DS07]{DinhSibony_pullback}
Tien-Cuong Dinh and Nessim Sibony.
\newblock Pull-back of currents by holomorphic maps.
\newblock {\em Manuscripta Math.}, 123(3):357--371, 2007.

\bibitem[DS14]{dinh-sibony:density}
Tien-Cuong Dinh and Nessim Sibony.
\newblock Density of positive closed currents, a theory of non-generic
  intersections.
\newblock {\em J. Alg. Geom., to appear. \url{http://arxiv.org/abs/1203.5810}},
  2014.

\bibitem[FS95]{Fornaess_Sibony}
John~Erik Forn{\ae}ss and Nessim Sibony.
\newblock Oka's inequality for currents and applications.
\newblock {\em Math. Ann.}, 301(3):399--419, 1995.

\bibitem[Kli91]{klimek}
Maciej Klimek.
\newblock {\em Pluripotential theory}, volume~6 of {\em London Mathematical
  Society Monographs. New Series}.
\newblock The Clarendon Press, Oxford University Press, New York, 1991.
\newblock Oxford Science Publications.

\end{thebibliography}
\end{document}